\documentclass[12pt, a4paper, twoside]{article}

\usepackage[margin=3cm]{geometry}

\usepackage{amsmath}%
\usepackage{amsfonts}%
\usepackage{amsthm}
\usepackage{mathtools}

\usepackage{paralist}

\let\epsilon\varepsilon

\let\phi\varphi

\let\theta\vartheta

\DeclareMathOperator{\realdel}{Re} \renewcommand{\Re}{\realdel}
\DeclareMathOperator{\imaginaerdel}{Im}
\DeclareMathOperator{\supp}{supp} \renewcommand{\Im}{\imaginaerdel}
\providecommand{\C}{\mathbb{C}} \providecommand{\N}{\mathbb{N}}
\providecommand{\R}{\mathbb{R}} \providecommand{\cH}{\mathcal{H}}
\providecommand{\cD}{\mathcal{D}} \providecommand{\cB}{\mathcal{B}}
\newtheorem{theorem}{Theorem}

\newtheorem{lemma}[theorem]{Lemma}
\newtheorem{proposition}[theorem]{Proposition}
\theoremstyle{definition}
\newtheorem{definition}[theorem]{Definition}
\newtheorem{remark}[theorem]{Remark}

\providecommand{\kdi}{\delta_i}
\providecommand{\kdj}{\delta_j}
\providecommand{\kdl}{\delta_\ell}

\DeclareMathOperator{\adjungeret}{ad}
\providecommand{\ad}[2]{\adjungeret_{#1}^{#2}}

\providecommand{\abs}[2][]{#1\lvert#2#1\rvert}
\providecommand{\norm}[2][]{#1\lVert#2#1\rVert}
\providecommand{\jnorm}[2][]{#1\langle#2#1\rangle}

\usepackage{microtype}
\usepackage[T1]{fontenc}
\usepackage[latin1]{inputenc}

\begin{document}
\title{A Taylor-like Expansion of a Commutator with a Function of
  Self-adjoint, Pairwise Commuting Operators}
\author{Morten Grud Rasmussen\\Department of Mathematical Sciences\\
  Aarhus University\\
  DK-8000 Aarhus\\
  Denmark\\
  \emph{email:} \texttt{mgr@imf.au.dk}}




\date{} 
\maketitle

\begin{abstract}Let $A$ be a $\nu$-vector of self-adjoint, pairwise
  commuting operators and $B$ a bounded operator of class
  $C^{n_0}(A)$. We prove a Taylor-like expansion of the commutator
  $[B,f(A)]$ for a large class of functions
  $f\colon\mathbb{R}^\nu\to\mathbb{R}$, generalising the
  one-dimensional result where $A$ is just a self-adjoint
  operator. This is done using almost analytic extensions and the
  higher-dimensional Helffer-Sj\"o{}strand formula.
  


\end{abstract}
\newpage
\section{Introduction}
It is well-known that if $A$ is a self-adjoint operator, $B$ is a
bounded operator of class $C^{n_0}(A)$ in the sense of \cite{ABG} and
$f$ satisfies $\abs{f^{(n)}(x)}\le C_n\jnorm{x}^{s-n}$ for all $n$,
then for $0\le t_1\le n_0$, $0\le t_2\le 1$ with $s+t_1+t_2<n_0$,
\begin{align*}
  [B,f(A)]=\sum_{k=1}^{n_0-1}\frac{1}{k!}f^{(k)}(A)\ad{A}{k}(B)+R_{n_0}(A,B)
\end{align*}
where $\ad{A}{k}(B)$ is the $k$'th iterated commutator,
$R_{n_0}(A,B)\in\cB(\cH_A^{-t_2};\cH_A^{t_1})$ and $\cH_A^t$ is
defined as $\cD(\jnorm{A}^t)$ equipped with the graph-norm
$\norm{v}_t=\norm{\jnorm{A}^tv}$ for $t\ge0$ and $\cH_A^{-t}$ is the
dual space of $\cH_A^t$. This follows relatively easily from using the
(one-dimensional) Helffer-Sj\"o{}strand formula
\begin{align}\label{eq:HS}
  f(A)=\frac{1}{\pi}\int_{\C}\bar\partial\tilde f(z)(A-z)^{-1}dz,
\end{align}
where $\bar\partial=\frac{1}{2}(\partial_x+i\partial_y)$ and $\tilde
f$ is an almost analytic extension of $f$, and the identity
\begin{align*}
  [B,f(A)] ={} &\sum_{k=1}^{n_0-1}\frac{1}{k!}
  \frac{k!}{\pi}\int_\C\bar\partial\tilde
  f(z)(-1)^k(A-z)^{-k-1}dz\\
  &+\frac{(-1)^{n_0}}{\pi}\int_\C\bar\partial\tilde
  f(z)(A-z)^{-n_0}\ad{A}{n_0}(B)(A-z)^{-1}dz
\end{align*}
when $\frac{k!}{\pi}\int_\C\bar\partial\tilde
f(z)(-1)^k(A-z)^{-k-1}dz$ is recognised as $f^{(k)}(A)$ using
\eqref{eq:HS}.  Such commutator expansions where first proved in
\cite{SiSo}. See e.g.\ \cite{MRMP1} for details. Due to the higher
complexity of the general Helffer-Sj\"o{}strand formula, these
calculations do not lead directly to the generalised result where $A$
is a vector of self-adjoint, pairwise commuting operators. However, we
will follow the same idea.

The theorem may be viewed as an abstract analogue of
pseudo-differential calculus. The one-dimensional version is an often
used result, see e.g. \cite{DGeBook} and \cite{MRMP1}. Apart from the
obvious interest in generalising the result to higher dimensions, our
improvement has proven useful in the treatment of models in quantum
field theory, see \cite{MGRJSM}. In particular, a lemma in
\cite{MGRJSM} whose proof depends on our result, extends the results
of \cite{MAHP} to a larger class of models.
\section{The setting and result}
In the following, $A=(A_1,\dotsc,A_\nu)$ is a vector of self-adjoint,
pairwise commuting operators acting on a Hilbert space $\cH$, and
$B\in\cB(\cH)$ is a bounded operator on $\cH$. We shall use the notion
of $B$ being of class $C^{n_0}(A)$ introduced in \cite{ABG}. For
notational convenience, we adobt the following convention: If $0\le
j\le\nu$, then $\kdj$ denotes the multi-index
$(0,\dotsc,0,1,0,\dotsc,0)$, where the $1$ is in the $j$'th entry.
\begin{definition}
  Let $n_0\in\N\cup\{\infty\}$. Assume that the mul\-ti-com\-mu\-ta\-tor form
  defined iteratively by $\ad{A}{0}(B)=B$ and
  $\ad{A}{\alpha}(B)=[\ad{A}{\alpha-\kdj}(B),A_j]$ as a form on
  $\cD(A_j)$, where $\alpha\ge\kdj$ is a multi-index and $1\le
  j\le\nu$, can be represented by a bounded operator also denoted by
  $\ad{A}{\alpha}(B)$, for all multi-indices $\alpha$,
  $\abs{\alpha}<n_0+1$. Then $B$ is said to be of class $C^{n_0}(A)$
  and we write $B\in C^{n_0}(A)$. 
\end{definition}
\begin{remark}
  The definition of $\ad{A}{\alpha}(B)$ does not depend on the order
  of the iteration since the $A_j$ are pairwise commuting. We call
  $\abs\alpha$ the \emph{degree} of $\ad{A}{\alpha}(B)$.
\end{remark}
In the following, $\cH_A^s := D(\abs{H}^s)$ for $s\ge0$ will be used
to denote the scale of spaces associated to $A$. For negative $s$, we
define $\cH_A^s := {(\cH_A^{-s})}^*$.

\begin{theorem}\label{thm}
  Assume that $B\in C^{n_0}(A)$ for some $n_0\ge n+1\ge 1$, $0\le t_1\le n+1$, $0\le
t_2\le 1$ and that $\{f_\lambda\}_{\lambda\in I}$
  satisfies 
  \begin{equation*}
    \label{eq:AAE}
    \forall\alpha\,\exists C_\alpha\colon\abs{\partial^\alpha
      f_\lambda(x)}\le C_\alpha\jnorm{x}^{s-\abs\alpha}
  \end{equation*}
  uniformly in $\lambda$ for some $s\in\R$ such that
  $t_1+t_2+s<n+1$. Then
  \begin{align*}
    [B,f_\lambda(A)]=\sum_{\abs\alpha=1}^n\frac{1}{\alpha!}\partial^\alpha
    f_\lambda(A)\,\ad{A}{\alpha}(B)+R_{\lambda,n}(A,B)
  \end{align*}
  as an identity on $\cD(\jnorm{A}^s)$, where
  $R_{\lambda,n}(A,B)\in\cB(\cH_A^{-t_2},\cH_A^{t_1})$ and there exist
  a constant $C$ independent of $A$, $B$ and $\lambda$ such that
  \begin{equation*}
    \norm{R_{\lambda,n}(A,B)}_{\cB(\cH_A^{-t_2},\cH_A^{t_1})}\le
    C\smashoperator{\sum_{\abs{\alpha}=n+1}}\norm{\ad{A}{\alpha}(B)}.
\end{equation*}
\end{theorem}
\begin{remark}
  A similar statement holds with the $\ad{A}{\alpha}(B)$ and
  $\partial^\alpha f_\lambda(A)$ interchanged at the cost of a sign
  correction given by $(-1)^{\abs\alpha-1}$, and the corresponding
  remainder term
  $R'_{\lambda,n}(A,B)\in\cB(\cH_A^{-t_1},\cH_A^{t_2})$. This can be
  seen either by proving it analogously or by taking the adjoint
  equation and replacing $B$ by $-B$. 
\end{remark}
\begin{remark}
  If $k\le t_1$ and $n_0\ge n+1+k$, then $R_{\lambda,n}(A,B)$ can be
  replaced by
  $R_{\lambda,n}^k(A,B)\in\cB(\cH_A^{-t_2+k},\cH_A^{t_1-k})$. This can
  be seen by commuting $\abs{A-z}^{-2}$ and $\ad{A}{\alpha}(B)$ in the
  terms of the remainder, see page~\pageref{remaindersiden}.
\end{remark}
\section{The Proof}
Let $z\in\C^\nu$, $\Im{z}\ne0$, $1\le\ell\le\nu$ and
$g,g_\ell\colon\R^\nu\to\C$ be given as $g(t)=\abs{t-z}^{-2}$ and
$g_\ell(t)=t_\ell-\bar z_\ell$. Write for $2\beta\le\alpha$
\begin{equation*}
  T_\alpha^\beta(t,z):=\tfrac{(-2)^{\abs{\alpha-\beta}}\abs{\alpha-\beta}!}{2^{\abs{\beta}}\beta!(\alpha-2\beta)!}(t-\Re z)^{\alpha-2\beta}\abs{t-z}^{-2\abs{\alpha-\beta}}.
\end{equation*}
\begin{lemma}\label{lem:lemma0} Let $g$ be as above and $\alpha$ be
  any multi-index. Then
  \begin{equation*}
    \partial^\alpha g(t)=\sum_{2\beta\le\alpha}\alpha!T_\alpha^\beta(t,z)\abs{t-z}^{-2}.
  \end{equation*}
\end{lemma}
\begin{proof}
  For brevity, we will write $\alpha^i$ or $\beta^i$ for
  $\alpha+\delta_i$ or $\beta+\delta_i$, respectively. The formula is
  obviously true for $\abs{\alpha}\le1$. Now assume that we have
  proven the formula for $\abs{\alpha}\le k$. Let $\abs{\alpha}=k$ and
  $0\le i\le\nu$ be arbitrary. It suffices to prove the formula for
  $\alpha^i$. One easily verifies using the chain rule that
  \begin{equation}
    (\partial^{\delta_i}g^n)(t)=-2n(t_i-\Re z_i)\abs{t-z}^{-2n-2}.\label{eq:chainrule}
  \end{equation}
  Now by the induction hypothesis, we see that
  \begin{align}
    \partial^{\alpha+\delta_i}g(t)={} &\partial_t^{\delta_i}\smashoperator{\sum_{2\beta\le\alpha}}\tfrac{(-2)^{\abs{\alpha-\beta}}\alpha!\abs{\alpha-\beta}!}{2^{\abs\beta}\beta!(\alpha-2\beta)!}(t-\Re
    z)^{\alpha-2\beta}\abs{t-z}^{-2\abs{\alpha-\beta}-2}\nonumber\\
   ={} &\smashoperator{\sum_{2\beta\le\alpha}}\tfrac{(-2)^{\abs{\alpha-\beta}}\alpha!\abs{\alpha-\beta}!}{2^{\abs\beta}\beta!(\alpha-2\beta)!}(\partial_t^{\delta_i}(t-\Re
    z)^{\alpha-2\beta})\abs{t-z}^{-2\abs{\alpha-\beta}-2}\label{eq:dumformel1}\\
    &+\smashoperator{\sum_{2\beta\le\alpha}}\tfrac{(-2)^{\abs{\alpha-\beta}}\alpha!\abs{\alpha-\beta}!}{2^{\abs\beta}\beta!(\alpha-2\beta)!}(t-\Re
    z)^{\alpha-2\beta}(\partial_t^{\delta_i}\abs{t-z}^{-2\abs{\alpha-\beta}-2}).\label{eq:dumformel2}
  \end{align}
  For the sake of clarity, we will now consider each sum
  independently.
  \begin{align}
    \eqref{eq:dumformel1}&=\smashoperator{\sum_{2\beta\le\alpha}}\tfrac{(-2)^{\abs{\alpha-\beta}}\alpha!\abs{\alpha-\beta}!}{2^{\abs\beta}\beta!(\alpha-2\beta)!}(\alpha_i-2\beta_i)(t-\Re
    z)^{\alpha-2\beta-\delta_i}\abs{t-z}^{-2\abs{\alpha-\beta}-2}\nonumber\\
    &=\smashoperator{\sum_{\substack{2\beta\le\alpha\\2\beta_i<\alpha_i}}}2(\beta_i+1)\tfrac{(-2)^{\abs{\alpha^i-\beta^i}}\alpha!\abs{\alpha^i-\beta^i}!}{2^{\abs{\beta^i}}\beta^i!(\alpha^i-2\beta^i)!}(t-\Re
    z)^{\alpha^i-2\beta^i}\abs{t-z}^{-2\abs{\alpha^i-\beta^i}-2}\nonumber\\
    &=\smashoperator{\sum_{2\beta\le\alpha+\delta_i}}2\beta_i\tfrac{(-2)^{\abs{\alpha^i-\beta}}\alpha!\abs{\alpha^i-\beta}!}{2^{\abs{\beta}}\beta!(\alpha^i-2\beta)!}(t-\Re
    z)^{\alpha^i-2\beta}\abs{t-z}^{-2\abs{\alpha^i-\beta}-2}.\label{eq:2beta1}
  \end{align}
  Using \eqref{eq:chainrule}, we see that \eqref{eq:dumformel2} equals
  \begin{align}
    \MoveEqLeft[3]
    \smashoperator{\sum_{2\beta\le\alpha}}\tfrac{(-2)^{\abs{\alpha-\beta}}\alpha!\abs{\alpha-\beta}!}{2^{\abs\beta}\beta!(\alpha-2\beta)!}(t-\Re 
    z)^{\alpha-2\beta}(-2)(\abs{\alpha-\beta}+1)(t_i-\Re
    z_i)\abs{t-z}^{-2\abs{\alpha-\beta}-4}\nonumber\\
    ={} &\smashoperator{\sum_{2\beta\le\alpha}}(\alpha_i+1-2\beta_i)\tfrac{(-2)^{\abs{\alpha^i-\beta}}\alpha!\abs{\alpha^i-\beta}!}{2^{\abs\beta}\beta!(\alpha^i-2\beta)!}(t-\Re
    z)^{\alpha^i-2\beta}\abs{t-z}^{-2\abs{\alpha^i-\beta}-2}\nonumber\\
    ={} &\smashoperator{\sum_{2\beta\le\alpha}}\tfrac{(-2)^{\abs{\alpha^i-\beta}}\alpha^i!\abs{\alpha^i-\beta}!}{2^{\abs\beta}\beta!(\alpha^i-2\beta)!}(t-\Re
    z)^{\alpha^i-2\beta}\abs{t-z}^{-2\abs{\alpha^i-\beta}-2}\label{eq:2betamindreendalpha}\\
    &-\smashoperator{\sum_{2\beta\le\alpha}}2\beta_i\tfrac{(-2)^{\abs{\alpha^i-\beta}}\alpha!\abs{\alpha^i-\beta}!}{2^{\abs\beta}\beta!(\alpha^i-2\beta)!}(t-\Re
    z)^{\alpha^i-2\beta}\abs{t-z}^{-2\abs{\alpha^i-\beta}-2}.\label{eq:2beta2}
  \end{align}
  Now \eqref{eq:2beta2} cancels \eqref{eq:2beta1} except for possible
  terms with $2\beta=\alpha+\delta_i$:
  \begin{equation}
    \label{eq:2betaligalphai}
    \eqref{eq:2beta1}+\eqref{eq:2beta2}=\smashoperator{\sum_{2\beta=\alpha+\delta_i}}\tfrac{(-2)^{\abs{\alpha^i-\beta}}\alpha^i!\abs{\alpha^i-\beta}!}{2^{\abs\beta}\beta!(\alpha^i-2\beta)!}(t-\Re
    z)^{\alpha^i-2\beta}\abs{t-z}^{-2\abs{\alpha^i-\beta}-2}.
  \end{equation}
  Adding \eqref{eq:2betamindreendalpha} and \eqref{eq:2betaligalphai}
  finishes the induction.
\end{proof}
\begin{lemma}\label{lem:lemma1}
  Let $B\in C^{n_0}(A)$ for some $n_0\ge1$ and let $n\in\N_0$ and
  $\alpha_0$ be a multi-index satisfying $\abs{\alpha_0}+n+1\le
  n_0$. Then
  \begin{equation}
    [\ad{A}{\alpha_0}(B),g(A)]=\label{eq:kommutator}
\sum_{\abs\alpha=1}^n\frac{1}{\alpha!}
    \partial^\alpha g(A)
    \ad{A}{\alpha_0+\alpha}(B)+R_n^g(A,\ad{A}{\alpha_0}(B)),
  \end{equation}
  where
  \begin{align}
    \MoveEqLeft[3] R_n^g(A,\ad{A}{\alpha_0}(B))\nonumber\\
={} &\smashoperator[l]{\sum_{\substack{\abs\alpha=n-1\\2\beta\le\alpha}}}\sum_{i=1}^\nu
    \tfrac{\beta_i+1}{\abs{\alpha+\kdi-\beta}}
    T_{\alpha+2\kdi}^{\beta+\kdi}(A,z)\ad{A}{\alpha_0+\alpha+2\kdi}(B)\abs{A-z}^{-2}
    \label{eq:restled1}\\
    &+\smashoperator[l]{\sum_{\substack{\abs\alpha=n\\2\beta\le\alpha}}}\sum_{i=1}^\nu
    \tfrac{\beta_i+1}{\abs{\alpha+\kdi-\beta}}
    T_{\alpha+2\kdi}^{\beta+\kdi}(A,z)(A_i-\bar
    z_i)\ad{A}{\alpha_0+\alpha+\kdi}(B)\abs{A-z}^{-2}
    \label{eq:restled2}\\
    &+\smashoperator[l]{\sum_{\substack{\abs\alpha=n\\2\beta\le\alpha}}}\sum_{i=1}^\nu
    \tfrac{\beta_i+1}{\abs{\alpha+\kdi-\beta}}
    T_{\alpha+2\kdi}^{\beta+\kdi}(A,z)\ad{A}{\alpha_0+\alpha+\kdi}(B)
    (A_i-z_i)\abs{A-z}^{-2}.
    \label{eq:restled3}
  \end{align}
\end{lemma}
\begin{proof} The proof goes by induction. One may check by inspection
  of the following identity that the statement is true for $n=0$.
\begin{equation}
  \begin{split}
    [\ad{A}{\alpha_0}(B),\abs{A-z}^{-2}]={} &-\sum_{i=1}^\nu\abs{A-z}^{-2}(A_i-\bar z_i)\ad{A}{\alpha_0+\delta_i}(B)\abs{A-z}^{-2}\\
    &-\sum_{i=1}^\nu\abs{A-z}^{-2}\ad{A}{\alpha_0+\delta_i}(B)(A_i-z_i)\abs{A-z}^{-2}.
\end{split}\label{eq:basestep}
\end{equation}
Now assume that we have proven the formula for $k\le n$,
$\abs{\alpha_0}+n+2\le n_0$. We will now show that this implies that
the formula holds for $k=n+1$. We begin by noting two useful identities.
\begin{equation}\label{eq:h1}
  T_\alpha^\beta(t,z)\abs{t-z}^{-2}=-\tfrac{\beta_j+1}{\abs{\alpha+\delta_j-\beta}}T_{\alpha+2\delta_j}^{\beta+\delta_j}(t,z).
\end{equation}
\begin{equation}
  \label{eq:h2}
  (\beta_i+1)T_{\alpha+2\delta_i}^{\beta+\delta_i}(t,z)2(t_i-\Re z_i)=(\alpha_i+1-2\beta_i)T_{\alpha+\delta_i}^\beta(t,z).
\end{equation}
Now using \eqref{eq:basestep} and \eqref{eq:h1} we see that
\begin{align}
  \eqref{eq:restled1}={} &\smashoperator[l]{\sum_{\abs{\alpha}=n-1}}\sum_{2\beta\le\alpha}\sum_{i=1}^\nu
  \tfrac{\beta_i+1}{\abs{\alpha+\kdi-\beta}}
  T_{\alpha+2\kdi}^{\beta+\kdi}(A,z)\abs{A-z}^{-2}\ad{A}{\alpha_0+\alpha+2\kdi}(B)\label{eq:restled11}\\
  \begin{split}&+\smashoperator[l]{\sum_{\abs{\alpha}=n-1}}\sum_{2\beta\le\alpha}\sum_{i=1}^\nu\sum_{j=1}^\nu
    \tfrac{\beta_i+1}{\abs{\alpha+\kdi-\beta}}\tfrac{\beta_j+\delta_{ij}+1}{\abs{\alpha+\delta_i+\delta_j-\beta}}
    T_{\alpha+2\kdi+2\kdj}^{\beta+\kdi+\kdj}(A,z)\\
    &\qquad\times(A_j-\bar z_j)\ad{A}{\alpha_0+\alpha+2\kdi+\kdj}(B)\abs{A-z}^{-2}
  \end{split}\label{eq:restled12}\\
  \begin{split}&+\smashoperator[l]{\sum_{\abs{\alpha}=n-1}}\sum_{2\beta\le\alpha}\sum_{i=1}^\nu\sum_{j=1}^\nu
    \tfrac{\beta_i+1}{\abs{\alpha+\kdi-\beta}}\tfrac{\beta_j+\delta_{ij}+1}{\abs{\alpha+\delta_i+\delta_j-\beta}}
    T_{\alpha+2\kdi+2\kdj}^{\beta+\kdi+\kdj}(A,z)\\
    &\qquad\times\ad{A}{\alpha_0+\alpha+2\kdi+\kdj}(B)(A_j-z_j)\abs{A-z}^{-2},
  \end{split}\label{eq:restled13}
\end{align}
and by reordering and reindexing the sum in \eqref{eq:restled11},
\eqref{eq:restled12} and \eqref{eq:restled13}, we get
\begin{equation}
  \eqref{eq:restled11}=\sum_{i=1}^\nu{\sum_{\substack{\abs{\alpha}=n+1\\\alpha_i\ge2}}}\sum_{\substack{2\beta\le\alpha\\\beta_i\ge1}}
  \tfrac{\beta_i}{\abs{\alpha-\beta}}T_{\alpha}^{\beta}(A,z)\abs{A-z}^{-2}\ad{A}{\alpha_0+\alpha}(B),\label{eq:restled11b}
\end{equation}
and \eqref{eq:restled12} equals
\begin{equation}
\sum_{i=1}^\nu{\sum_{\substack{\abs{\alpha}=n+1\\\alpha_i\ge2}}}\sum_{\substack{2\beta\le\alpha\\\beta_i\ge1}}\sum_{j=1}^\nu
    \tfrac{\beta_i}{\abs{\alpha-\beta}}\tfrac{\beta_j+1}{\abs{\alpha+\delta_j-\beta}}
    T_{\alpha+2\kdj}^{\beta+\kdj}(A,z) (A_j-\bar
    z_j)\ad{A}{\alpha_0+\alpha+\kdj}(B)\abs{A-z}^{-2}
\label{eq:restled12b}
\end{equation}
and similarly for \eqref{eq:restled13} with the  factor
$(A_j-\bar
z_j)\ad{A}{\alpha_0+\alpha+\kdj}(B)$ replaced by the factor 
$\ad{A}{\alpha_0+\alpha+\kdj}(B)(A_j-\nobreak z_j)$. Note that we may relax the
extra conditions on $\alpha$ and $\beta$ in the above statements, as a
term with $\beta_i=0$ contributes nothing.

Instead of continuing in the same fashion with \eqref{eq:restled2} and
\eqref{eq:restled3}, we note using \eqref{eq:h2} that
\begin{align}
  \eqref{eq:restled2}+\eqref{eq:restled3}
={} &\smashoperator[l]{\sum_{\abs{\alpha}=n}}\sum_{2\beta\le\alpha}\sum_{i=1}^\nu\tfrac{\beta_i+1}{\abs{\alpha+\kdi-\beta}}
T_{\alpha+2\kdi}^{\beta+\kdi}(A,z)\ad{A}{\alpha_0+\alpha+2\kdi}(B)
\abs{A-z}^{-2}\label{eq:restled231}\\
  &+\smashoperator[l]{\sum_{\abs{\alpha}=n}}\sum_{2\beta\le\alpha}\sum_{i=1}^\nu\tfrac{\alpha_i+1-2\beta_i}{\abs{\alpha+\kdi-\beta}}
T_{\alpha+\kdi}^{\beta}(A,z)\ad{A}{\alpha_0+\alpha+\kdi}(B)\abs{A-z}^{-2}
,\label{eq:restled232}
\end{align}
so we may focus our attention on \eqref{eq:restled232}:
\begin{align}
  \eqref{eq:restled232}= {} &
  \sum_{i=1}^\nu\sum_{\substack{\abs{\alpha}=n+1\\\alpha_i\ge 1}}
  \sum_{\substack{2\beta\le\alpha\\2\beta_i<\alpha_i}}\tfrac{\alpha_i-2\beta_i}{\abs{\alpha-\beta}}T_{\alpha}^\beta(A,z)\abs{A-z}^{-2}\ad{A}{\alpha_0+\alpha}(B)\label{eq:restled2321}\\
  & \begin{aligned}
    +&\sum_{i=1}^\nu\sum_{\substack{\abs{\alpha}=n+1\\\alpha_i\ge 1}}
  \sum_{\substack{2\beta\le\alpha\\2\beta_i<\alpha_i}}\sum_{j=1}^\nu\tfrac{\alpha_i-2\beta_i}{\abs{\alpha-\beta}}\tfrac{\beta_j+1}{\abs{\alpha+\delta_j-\beta}}T_{\alpha+2\kdj}^{\beta+\kdj}(A,z)\\
  &\quad\times(A_j-\bar z_j)\ad{A}{\alpha_0+\alpha+\kdj}(B)\abs{A-z}^{-2}.
  \end{aligned}\label{eq:restled2323}\\
  &\begin{aligned}
    +&\sum_{i=1}^\nu\sum_{\substack{\abs{\alpha}=n+1\\\alpha_i\ge 1}}
  \sum_{\substack{2\beta\le\alpha\\2\beta_i<\alpha_i}}\sum_{j=1}^\nu\tfrac{\alpha_i-2\beta_i}{\abs{\alpha-\beta}}\tfrac{\beta_j+1}{\abs{\alpha+\delta_j-\beta}}T_{\alpha+2\kdj}^{\beta+\kdj}(A,z)\\
  &\quad\times\ad{A}{\alpha_0+\alpha+\kdj}(B)(A_j-z_j)\abs{A-z}^{-2}
  \end{aligned}\label{eq:restled2322}
\end{align}
We note again that the additional conditions on $\alpha$ and $\beta$
are superfluous. 

We may now recollect the terms. First we see using
Lemma~\ref{lem:lemma0}:
\begin{equation}
    \sum_{\abs{\alpha}=1}^n\frac{1}{\alpha!}\partial^\alpha
    g(A)\ad{A}{\alpha_0+\alpha}(B)+\eqref{eq:restled11b}+
    \eqref{eq:restled2321}=\sum_{\abs{\alpha}=1}^{n+1}\frac{1}{\alpha!}\partial^\alpha
    g(A)\ad{A}{\alpha_0+\alpha}(B),\label{eq:isum}
\end{equation}
then
\begin{equation}
  \eqref{eq:restled12b}+\eqref{eq:restled2323}=
  \smashoperator[l]{\sum_{\substack{\abs{\alpha}=n+1\\2\beta\le\alpha}}}
  \sum_{j=1}^\nu\tfrac{\beta_j+1}{\abs{\alpha+\delta_j-\beta}}T_{\alpha+2\kdj}^{\beta+\kdj}(A,z)(A_j-\bar
  z_j)\ad{A}{\alpha_0+\alpha+\kdj}(B)\abs{A-z}^{-2}\label{eq:irestled2},
\end{equation}
and
\begin{equation}
    \eqref{eq:restled13}+\eqref{eq:restled2322}=\smashoperator[l]{\sum_{\substack{\abs{\alpha}=n+1\\2\beta\le\alpha}}}
  \sum_{j=1}^\nu\tfrac{\beta_j+1}{\abs{\alpha+\delta_j-\beta}}T_{\alpha+2\kdj}^{\beta+\kdj}(A,z)\ad{A}{\alpha_0+\alpha+\kdj}(B)(A_j-
  z_j)\abs{A-z}^{-2}\label{eq:irestled3},
\end{equation}
so adding up, we have proved that \eqref{eq:kommutator} equals the sum
of \eqref{eq:isum}, \eqref{eq:restled231}, \eqref{eq:irestled2} and
\eqref{eq:irestled3} as stated.
\end{proof}

The following lemma plays the same role for $g_\ell$ as
Lemma~\ref{lem:lemma1} plays for $g$, but contrary to
Lemma~\ref{lem:lemma1}, the proof is trivial.
\begin{lemma}\label{lem:lemma2}Let $B\in C^{n_0}(A)$ for some $n_0\ge1$ and let $n\in\N_0$ and
  $\alpha_0$ be a multi-index satisfying $\abs{\alpha_0}+n+1\le
  n_0$. Then
  \begin{align*}
    [\ad{A}{\alpha_0}(B),g_\ell(A)]
    &=\sum_{\abs\alpha=1}^n \frac{1}{\alpha!} \partial^\alpha
    g_\ell(A)\ad{A}{\alpha_0+\alpha}(B)+R_n^{g_\ell}(A,\ad{A}{\alpha_0}(B)),
  \end{align*}
  where $R_n^{g_\ell}(A,\ad{A}{\alpha_0}(B))=0$ for $n\ge1$,
  $R_0^{g_\ell}(A,\ad{A}{\alpha_0}(B))=\ad{A}{\alpha_0+\kdl}(B)$.
\end{lemma}
The following lemma also follows by induction.
\begin{lemma}\label{lem:1}
  Let $B\in C^{n_0}(A)$ for some $n_0\ge1$. Assume that
  $h_i\in C^\infty(\R^\nu)$, $1\le i\le k$, satisfies
  \begin{equation*}
    [\ad{A}{\alpha_0}(B),h_i(A)]=\sum_{\abs\alpha=1}^n
    \frac{1}{\alpha!} \partial^\alpha
    h_i(A)\ad{A}{\alpha_0+\alpha}(B)+R_n^{h_i}(A,\ad{A}{\alpha_0}(B)),
  \end{equation*}
  where $R_n^{h_i}(A,\ad{A}{\alpha_0}(B))$ is bounded for all $n\in\N_0$
  and multi-indices $\alpha_0$ satisfying $\abs{\alpha_0}+n+1\le n_0$
  and $\partial^\alpha h_i(A)$ is bounded for all $1\le\abs{\alpha}\le
  n_0-1$. Then
\begin{align*}
  \MoveEqLeft\Bigl[B,\prod_{i=1}^k h_i(A)\Bigr]=\sum_{\abs{\alpha}=1}^n
  \frac{1}{\alpha!}\partial^\alpha \Bigl(\prod_{i=1}^k h_i\Bigr)
  (A)\ad{A}{\alpha}(B)
  \\
  &+\sum_{j=1}^k\sum_{\abs{\alpha}=0}^n\frac{1}{\alpha!}\partial^\alpha
  \Bigl(\prod_{i=1}^{j-1}h_i\Bigr)
  (A)R_{n-\abs{\alpha}}^{h_j}(A,\ad{A}{\alpha}(B))\prod_{i=j+1}^k h_i(A).
\end{align*}
\end{lemma}

Let $n+1\le n_0$. If we put $k=\nu+1$, $h_i=g$ for $i\ne\nu$,
$h_{\nu}=g_\ell$ and apply Lemma~\ref{lem:lemma1}, \ref{lem:lemma2}
and \ref{lem:1} we see that
\begin{equation}
  \begin{split}
    \MoveEqLeft{} [B,\abs{A-z}^{-2\nu}(A_\ell-\bar z_\ell)]
    \\
    &=\sum_{\abs{\alpha}=1}^n \frac{1}{\alpha!}\partial^\alpha
    \bigl(\abs{\,\cdot-z}^{-2\nu}(\,\cdot\,_\ell-\bar z_\ell)\bigr)
    (A)\ad{A}{\alpha}(B)+R_{\ell,n}(A,B),
  \end{split}
  \label{eq:AcommB}
\end{equation}
where 
\begin{align}
\MoveEqLeft[3] R_{\ell,n}(A,B)\nonumber
\\\label{eq:restled4}
={} &\sum_{j=1}^{\nu-1}\sum_{\abs{\alpha}=0}^n\frac{1}{\alpha!}\partial^\alpha
  (g^{j-1})
  (A)R_{n-\abs{\alpha}}^{g}(A,\ad{A}{\alpha}(B))\abs{A-z}^{-2(\nu-j)}(A_\ell-\bar
  z_\ell)\\\label{eq:restled5}
  &+\sum_{\abs\alpha=n}\frac{1}{\alpha!}\partial^\alpha
  (g^{\nu-1})
  (A)\ad{A}{\alpha+\kdl}(B)\abs{A-z}^{-2}\\\label{eq:restled6}
  &+\sum_{\abs{\alpha}=0}^n\frac{1}{\alpha!}\partial^\alpha
  (g^{\nu-1}g_\ell)
  (A)R_{n-\abs{\alpha}}^{g}(A,\ad{A}{\alpha}(B))
\end{align}
In the following, we will refer to the terms of $R_{\ell,n}(A,B)$ as
the remainder terms. Let $0\le t_1\le n+1$ and $0\le t_2\le 1$. By Hadamard's three-line
lemma and using
(\ref{eq:restled1}--\ref{eq:restled3}),
(\ref{eq:restled4}--\ref{eq:restled6}), Lemma~\ref{lem:lemma0} and the
identity
\begin{equation*}
  \partial^\alpha\Bigl(\prod_{i=1}^j f_i\Bigr) =
  \smashoperator[l]{\sum_{\sum
      \alpha_i=\alpha}}\frac{\alpha!}{\prod_{i=1}^j\alpha_i!}\prod_{i=1}^j\partial^{\alpha_i}f_i,
\end{equation*}
we may inspect that each remainder term (with $R_{\ell,n}(A,B)$
replaced by the remainder term) and hence $R_{\ell,n}(A,B)$ satisfies
the inequality
\begin{align}\label{eq:Hadamard}
  \norm{\jnorm{A}^{t_1}R_{\ell,n}(A,B)\jnorm{A}^{t_2}}\le
  C\jnorm{z}^{t_1+t_2}\abs{\Im z}^{-n-2\nu}.
\end{align}\label{remaindersiden}
We will now use the functional calculus of almost analytic extensions.
See e.g.\ \cite{DiSj} for details. In the following, we write
$\bar\partial=(\bar\partial_1,\dotsc,\bar\partial_\nu)$ where
$\bar\partial_j=\frac12(\partial_{u_j}+i\partial_{v_j})$ and $u_j+v_j=z_j\in\C$,
$z=(z_1,\dotsc,z_n)\in\C^\nu$. The following proposition is inspired
by  \cite{MRMP1} and \cite[Chap.~X.2]{Tr}.
\begin{proposition}\label{prop:AA}
  Let $s\in\R$ and $\{f_\lambda\}_{\lambda\in I}\subset C^\infty(\R^\nu)$ satisfy
  \begin{align*}\label{eq:AAregularity}
    \forall \alpha\ \exists C_\alpha\colon\abs{\partial^\alpha
      f_\lambda(x)}\le C_\alpha\jnorm{x}^{s-\abs{\alpha}}.
  \end{align*}
  There exists a family of almost analytic extensions $\{\tilde
  f_\lambda\}_{\lambda\in I}\subset C^\infty(\C^\nu)$ satisfying
  \begin{enumerate}[(i)]
  \item $\supp(\tilde f_\lambda)\subset\{u+iv\mid
    u\in\supp(f_\lambda), \abs{v}\le C\jnorm{u}\}.$
  \item\label{item:AA2} $\forall \ell\ge0\ \exists C_\ell\colon\abs{\bar\partial\tilde
      f_\lambda(z)}\le C_\ell\jnorm{z}^{s-\ell-1}\abs{\Im z}^\ell$.
  \end{enumerate}
\end{proposition}
\begin{proof}
  We define a mapping $C^\infty(\R^\nu)\ni f\mapsto\tilde f\in
  C^\infty(\C^\nu)$ in the following way. Choose a function $\kappa\in
  C_0^\infty(\R)$ which equals $1$ in a neighbourhood of $0$ and put
  $\lambda_0=C_0$,
  $\lambda_k=\max\{\max_{\abs{\alpha}=k}C_\alpha,\lambda_{k-1}+1\}$
  for $k\ge1$. Writing $z=u+iv\in\R^\nu\oplus i\R^\nu$, we now define
  \begin{equation*}
    \tilde f(z)=\sum_\alpha\frac{\partial^\alpha
      f(u)}{\alpha!}(iu)^\alpha\prod_{j=1}^\nu\kappa\Bigl(\frac{\lambda_{\abs{\alpha}}v_j}{\jnorm{u}}\Bigr).
  \end{equation*}
  One can now check that the properties hold.
\end{proof}
\begin{remark}\label{rem:AA}
  Note that if we for a $\chi\in C_0^\infty(\R^\nu;[0,1])$ with
  $\chi(0)=1$ define a sequence of functions by
  $f_{k,\lambda}(x)=\chi(\frac{x}{k})f_\lambda(x)$, then
  \begin{equation*}
    [B,f_\lambda(A)]=\lim_{k\to\infty}[B,f_{k,\lambda}(A)]
  \end{equation*}
  as a form identity on $\cD(\jnorm{A}^s)$ and we have the dominated
  pointwise convergence
  \begin{equation*}
    \bar\partial \tilde f_{k,\lambda}(x)\to\bar\partial \tilde
    f_\lambda(x)\text{ for }k\to\infty.
  \end{equation*}
\end{remark}
Let $\{f_\lambda\}_{\lambda\in I}$ satisfy the assumption of
Proposition~\ref{prop:AA} with $s<0$. Then the almost analytic
extensions provide a functional calculus via the formula
\begin{align}
  f_\lambda(A)=C_{\nu}\sum_{\ell=1}^\nu\int_{\C^\nu}\bar\partial_\ell\tilde
  f_\lambda(z)(A_\ell-\bar z_\ell)\abs{A-z}^{-2\nu}dz,\label{eq:AAFC}
\end{align}
where $C_\nu$ is a positive constant (again we refer to \cite{DiSj}
for details). Note that the integrals are absolutely convergent by Proposition~\ref{prop:AA}\eqref{item:AA2}.

Multiplying $\jnorm{A}^{t_1}R_{\ell,n}(A,B)\jnorm{A}^{t_2}$ with
$\bar\partial\tilde f_\lambda(z)$, we get from \eqref{eq:Hadamard} and
Proposition~\ref{prop:AA}~\eqref{item:AA2} that
\begin{align}\label{eq:Rest}
  \norm{\jnorm{A}^{t_1}\bar\partial\tilde
    f_\lambda(z)R_{\ell,n}(A,B)\jnorm{A}^{t_2}}\le C\jnorm{z}^{t_1+t_2+s-n-1-2\nu}.
\end{align}
Hence, if $t_1+t_2+s<n+1$, $\jnorm{A}^{t_1}\bar\partial\tilde
f_\lambda(z)R_{\ell,n}(A,B)\jnorm{A}^{t_2}$ is integrable over
$\C^\nu$. Using \eqref{eq:AcommB}, \eqref{eq:AAFC} and
\eqref{eq:Rest}, we see that
\begin{align}
  [B,f_\lambda(A)]= {} & 
  C_{\nu}\sum_{\ell=1}^\nu\int_{\C^\nu}\bar\partial_\ell\tilde 
  f_\lambda(z)[B,(A_\ell-\bar z_\ell)\abs{A-z}^{-2\nu}]\,dz\nonumber\\
  ={} &C_\nu\sum_{\ell=1}^\nu\int_{\C^\nu}\bar\partial_\ell\tilde
  f_\lambda(z)\smashoperator{\sum_{\abs{\alpha}=1}^n}\frac{1}{\alpha!}\partial^\alpha
  \bigl(\abs{\,\cdot-z}^{-2\nu}(\,\cdot\,_\ell-\bar z_\ell)\bigr)
  (A)\,dz\,\ad{A}{\alpha}(B)\nonumber\\
  &+C_\nu\sum_{\ell=1}^\nu\int_{\C^\nu}\bar\partial_\ell\tilde
  f_\lambda(z)R_{\ell,n}(A,B)\,dz.\label{eq:Restled}
\end{align}
We denote \eqref{eq:Restled} by $R_{\lambda,n}(A,B)$.  Note that
\begin{align*}
  \MoveEqLeft \sum_{\ell=1}^\nu\int_{\C^\nu}\bar\partial_\ell\tilde
  f_\lambda(z)
  \frac{1}{\alpha!}\partial_t^\alpha
  \bigl(\abs{t-z}^{-2\nu}(t_\ell-\bar z_\ell)\bigr)
  \,dz\\
  &=
  \frac{1}{\alpha!}\partial_t^\alpha\sum_{\ell=1}^\nu\int_{\C^\nu}\bar\partial_\ell\tilde
  f_\lambda(z) \abs{t-z}^{-2\nu}(t_\ell-\bar z_\ell) \,dz
=
\frac{1}{\alpha!}\partial^\alpha
  f_\lambda(t),
\end{align*}
which implies 
\begin{align*}
  [B,f_\lambda(A)]&=\sum_{\abs{\alpha}=1}^n\frac{1}{\alpha!}\partial^\alpha
  f_\lambda(A)\,\ad{A}{\alpha}(B)+R_{\lambda,n}(A,B).
\end{align*}

We have now proved Theorem~\ref{thm} in the case $s<0$.
For the general case, we use Remark~\ref{rem:AA} to see that
$[B,f_\lambda(A)]=\lim_{k\to\infty}[B,f_{k,\lambda}(A)]$ and clearly,
$f_{k,\lambda}$ satisfies the assumption of Proposition~\ref{prop:AA}
with the same $s$, so the estimate corresponding to \eqref{eq:Rest} is
now uniform in $k$ and $\lambda$. The pointwise convergence and
Lebesgue's theorem on dominated convergence now finishes the argument.
\section*{Acknowledgements}
The author would like to thank J.~S. M\o{}ller for suggestions,
fruitful discussions and ultimately for proposing this problem. Part
of this work was done while participating in the Summer School on
Current Topics in Mathematical Phyics at the Erwin Schr\"o{}dinger
International Institute for Mathematical Physics (ESI) in Vienna.
\bibliographystyle{plain}
\bibliography{referencer}

\begin{thebibliography}{1}

\bibitem{ABG}
W.~O. Amrein, A.~Boutet~de Monvel, and V.~Georgescu.
\newblock {\em {$C_0$}-Groups, Commutator Methods and Spectral Theory of
  {$N$}-body Hamiltonians}.
\newblock Birkh{\"a}user, 1996.

\bibitem{DGeBook}
J.~Derezi\'n{}ski and C.~G\'e{}rard.
\newblock {\em Scattering Theory of Classical and Quantum {$N$}-Particle
  systems}.
\newblock Texts and Monographs in Physics. Springer, Berlin, 1997.

\bibitem{DiSj}
M.~Dimassi and J.~Sj\"o{}strand.
\newblock {\em Spectral Asymptotics in the Semi-Classical Limit}, volume 268 of
  {\em London Mathematical Society Lecture Note Series}.
\newblock Cambridge University Press, 1999.

\bibitem{MRMP1}
J.~S. {M\o{}ller}.
\newblock An abstract radiation condition and application to {$N$}-body
  systems.
\newblock {\em Rev. Math. Phys.}, 12(5):767--803, 2000.

\bibitem{MAHP}
J.~S. M{\o}ller.
\newblock The translation invariant massive {Nelson} model: {I}. the bottom of
  the spectrum.
\newblock {\em Ann. Henri Poincar{\'e}}, 6:1091--1135, 2005.

\bibitem{MGRJSM}
J.~S. M\o{}ller and M.~G. Rasmussen.
\newblock The translation invariant massive {Nelson} model: {II}. The continuous spectrum below the two-boson threshold.
\newblock Submitted.

\bibitem{SiSo}
I.~M. Sigal and A.~Soffer, \emph{The {$N$}-particle scattering problem:
  Asymptotic completeness for short-range quantum systems}, Ann. Math.
  \textbf{125} (1987), 35--108.

\bibitem{Tr}
F.~Treves.
\newblock {\em Introduction to Pseudodifferential Operators and Fourier
  Integral Operators}, volume~2.
\newblock Plenum Press, 1980.

\end{thebibliography}
\end{document}